\documentclass{article}
\usepackage{amsmath}
\usepackage{amsthm}
\usepackage{amssymb}

\theoremstyle{plain}

\newtheorem{thm}{Theorem}
\newtheorem{cor}[thm]{Corollary}
\newtheorem{lem}[thm]{Lemma}

\newtheorem{claim}[thm]{Claim}
\newtheorem{mainclaim}[thm]{Main Claim}
\newtheorem{mainlem}[thm]{Main Lemma}

\newtheorem*{theorem}{Theorem}

\theoremstyle{definition}

\theoremstyle{remark}

\newcounter{enuroman}
\renewcommand{\theenuroman}{\roman{enuroman}}

\newcounter{enuRoman}
\renewcommand{\theenuRoman}{(\Roman{enuRoman})}

\newcounter{enuAlph}
\renewcommand{\theenuAlph}{\Alph{enuAlph}}

\newcounter{enualph}
\renewcommand{\theenualph}{\alph{enualph}}

\newcounter{enuarabic}
\renewcommand{\theenuarabic}{\arabic{enuarabic}}

\newcommand{\forces}{\Vdash}
\newcommand{\re}{{\upharpoonright}}

\newcommand{\embed}{{<\!\!\circ\;}}

\newcommand{\CC}{{\mathbb C}}

\newcommand{\PP}{{\mathbb P}}

\newcommand{\cc}{{\mathfrak c}}

\newcommand{\PAT}{{\mathsf{PAT}}}

\newcommand{\sft}{{\mathsf{t}}}

\newcommand{\dom}{{\mathrm{dom}}}

\newcommand{\supp}{{\mathrm{supp}}}

\newcommand{\id}{{\mathrm{id}}}

\newcommand{\sub}{\subseteq}
\newcommand{\sem}{\setminus}
\newcommand{\twoom}{2^\omega}
\newcommand{\twolom}{2^{<\omega}}

\newcommand{\Loleriar}{\mbox{$\Longleftrightarrow$}}

\title{Q}

\author{J\"org Brendle\thanks{Partially supported by Grant-in-Aid for Scientific Research
   (C) 15K04977, Japan Society for the Promotion of Science. I also acknowledge partial 
   support from Michael Hru\v s\'ak's grants, CONACyT grant no. 177758  and
   PAPIIT grant IN-108014, during my stay at UNAM in spring 2015, when this research was started.}  \\
   Graduate School of System Informatics \\
   Kobe University \\
   Rokko-dai 1-1, Nada-ku \\
   Kobe 657-8501, Japan \\
   email: {\sf brendle@kurt.scitec.kobe-u.ac.jp}
}

\begin{document}
\maketitle

\begin{abstract}
\noindent A Q-set is an uncountable set of reals all of whose subsets are relative $G_\delta$ sets. 
We prove that, for an arbitrary uncountable cardinal $\kappa$, there is consistently a
Q-set of size $\kappa$ whose square is not Q. This answers a question of A. Miller.
\end{abstract}



\section{Introduction}

A {\em Q-set} is an uncountable set of reals in which every subset is a relative $G_\delta$ set.
If $X$ is a Q-set, then clearly $2^{|X|} = \cc$. In particular, the existence of Q-sets implies
$2^{\omega_1} = \cc$. On the other hand, under Martin's Axiom MA every uncountable
sets of reals of cardinality $< \cc$ is Q~\cite{MS70}. Przymusi\'nski~\cite{Pr80} proved that if there
is a Q-set, then there is a Q-set of size $\aleph_1$ all of whose finite powers are Q as well.
Furthermore, by the above, it is clear that under MA all finite powers of Q-sets are again Q.
This left open the question of whether there can consistently exist a Q-set whose square is not
Q. In~\cite{Fl83}, Fleissner claimed the consistency of a stronger statement, namely, that there
is a Q-set of size $\aleph_2$ while no square of a set of reals of size $\aleph_2$ is Q.
However, Miller~\cite[Theorem 8]{Mi07} observed that Fleissner's argument was flawed.
More explicitly, he showed that if there is a Q-set of size $\aleph_2$, then there is a
set of reals $X = \{ x_\alpha : \alpha < \omega_2 \}$ such that the set of points ``above
the diagonal", $\{ (x_\alpha, x _\beta ) : \alpha < \beta < \omega_2 \}$, is a
relative $G_\delta$ set in the square $X^2$. While this does not contradict the statement
of Fleissner's result (whose correctness remains open as far as we know), it does contradict
his method of proof, for he claimed that in his model, for every set of reals $X = \{ x_\alpha :
\alpha < \omega_2 \}$, the set $\{ (x_\alpha,x_\beta) : \alpha < \beta < \omega_2 \}$ was not
a relative $G_\delta$ in $X^2$.

Here we show:
\begin{theorem}
Let $\kappa$ be an arbitrary uncountable cardinal.
It is consistent that there is a Q-set of size $\kappa$ whose square is not Q.
\end{theorem}
This answers a question of A. Miller~\cite[Problem 7.16]{Mi15} (see also~\cite[after the proof of
Theorem 8 on p. 32]{Mi07}).

Our model is, in a sense, very natural: we first add $\kappa$ Cohen reals $c_\alpha $,
$\alpha < \kappa $, and then turn the set $C = \{ c_\alpha : \alpha < \kappa \}$ into a Q-set
in a finite support iteration of length $\kappa^+$, going through all subsets of $C$ by a
book-keeping argument, and turning them into relative $G_\delta$'s by ccc forcing
(see, e.g., \cite[Section 5]{Mi95} for this forcing). To have control over the names for conditions arising in the
iteration, we present a recursive definition of this forcing with finitary conditions. Incidentally,
this model is the same as the one used by Fleissner in~\cite{Fl83}, for the case $\kappa =\omega_2$,
though our description is somewhat different. Furthermore, it is similar to the one used by Fleissner and Miller in~\cite{FM80}, 
for the case $\kappa =\omega_1$, so that our technique also shows that the square of their Q-set, which
is concentrated on the rationals, is not a Q-set. The main point is, of course, to prove that
$C^2$ is not a Q-set. To this end, we show that the set of points ``above the diagonal",
$\{ (c_\alpha, c_\beta) : \alpha < \beta < \kappa \}$, is not a relative $G_\delta$ set in $C^2$.
Unlike for Fleissner's work~\cite{Fl83}, there is no contradiction with Miller's result~\cite[Theorem 8]{Mi07},
because we prove a much weaker statement. 

\bigskip

{\bf Acknowledgment.} I thank Michael Hru\v s\'ak for bringing this problem to my attention and for 
suggesting the model. I am also grateful to the referee for pointing out the reference~\cite{FM80}.

\section{Proof of the Theorem}

Assume the ground model $V$ satisfies GCH.
We perform a finite support iteration $( \PP_\gamma : \gamma \leq \kappa^+ )$ of ccc forcing.
As usual, elements of $\PP_\gamma$ are functions with domain $\gamma$.
Also, a book-keeping argument will hand us down a sequence $(\dot A_\gamma : \gamma < \kappa^+)$
of $\PP_\gamma$-names of subsets of $\kappa$ such that for every $\PP_{\kappa^+}$-name
$\dot A$ for a subset of $\kappa$ there are $\gamma < \kappa^+$ such that $\PP_\gamma$
forces that $\dot A = \dot A_\gamma$. Since this is a standard argument we will omit details.
The recursive definition of the $\PP_\gamma$ is as follows:
\begin{itemize}
\item $\PP_0$ is the trivial forcing (as usual).
\item $\PP_1$ is $\CC_{\kappa}$, the forcing for adding $\kappa$ Cohen reals:
   if $p \in \PP_1$, then $p(0) = ( \sigma_\alpha^p : \alpha \in F^p)$ where $F^p \sub \kappa$ is finite and
   $\sigma_\alpha^p \in \twolom$ for $\alpha \in F^p$; the ordering is given by 
   $q \leq p$ if $F^q \supseteq F^p$ and $\sigma_\alpha^q \supseteq \sigma_\alpha^p$ for all $\alpha \in F^p$.
\item If $\gamma$ is a limit ordinal, $\PP_\gamma$ consists of all functions $p$ with $\dom (p) = \gamma$
   such that there is $0<\delta < \gamma$ with $p\re\delta \in \PP_\delta$ and $p (\epsilon) = \emptyset$ for
   all $\epsilon$ with $\delta \leq\epsilon < \gamma$; the ordering is given by $q \leq p$ if
   $q \re \delta \leq p \re \delta$ for all $\delta < \gamma$.
\item If $\gamma \geq 1$, $\PP_{\gamma + 1}$ consists of all functions $p$ with $\dom (p) = \gamma + 1$,
   $p \re \gamma \in \PP_\gamma$ and $p(\gamma)$ is a finite (possibly empty) subset of $(\twolom \cup \kappa) \times \omega$
   such that
   \begin{itemize}
   \item if $(\alpha,n) \in p(\gamma)$ for some $\alpha \in \kappa$ and $n \in \omega$, then $\alpha \in F^p$ and 
      $p\re\gamma \forces \alpha \notin \dot A_\gamma$,
   \item if $(\alpha,n) , (\sigma,n) \in p(\gamma)$ for some $\alpha \in\kappa$, $\sigma\in\twolom$ and $n\in\omega$,
      then $\sigma^p_\alpha$ and $\sigma$ are incompatible in $\twolom$;
   \end{itemize}
   the ordering is given by $q \leq p$ if $q\re\gamma \leq p\re\gamma$ and $q(\gamma) \supseteq p(\gamma)$.
\end{itemize}
For $\gamma \leq \kappa^+$ and $p \in \PP_\gamma$, let $\supp (p) = \{ 0 \} \cup \{ \gamma > 0 : p(\gamma) \neq
\emptyset \}$ denote the {\em support} of $p$. First of all, let us note that this is indeed an iteration, that is,
$\PP_\gamma \embed \PP_\delta$ for $\gamma < \delta$. This is straightforward, for if $p \in \PP_\delta$ and
$q' \in \PP_\gamma$ with $q' \leq p \re \gamma$, then $q \in \PP_\delta$ defined by $q \re \gamma = q'$ and
$q (\epsilon) = p (\epsilon)$ for $\gamma \leq \epsilon <\delta$ is a common extension of $q'$ and $p$. 

Here is a sufficient criterion for compatibility of conditions. 

\begin{lem}  \label{compat-crit}
Assume $p,q \in \PP_\gamma$ are such that
\begin{itemize}
\item $\sigma_\alpha^p = \sigma_\alpha^q$ for all $\alpha \in F^p \cap F^q$,
\item  for all $\delta \in \supp (p) \cap \supp (q)$ with $\delta > 0$ and all $(\sigma, n) \in \twolom \times \omega$,
\[ (\sigma,n) \in p(\delta) \; \Loleriar \; (\sigma,n) \in q (\delta). \]
\end{itemize}
Then $p$ and $q$ are compatible with a canonical common extension $r$ given by 
\begin{itemize}
\item $\supp (r) = \supp (p) \cup \supp (q)$,
\item $F^r = F^p \cup F^q$,
\item $\sigma_\alpha^r = \begin{cases} \sigma_\alpha^p & \mbox{if } \alpha \in F^p \\ \sigma_\alpha^q & 
   \mbox{if } \alpha \in F^q, \\ \end{cases}$
\item $r(\delta) = p(\delta) \cup q (\delta)$ for all $\delta \in \supp (r)$ with $\delta > 0$.
\end{itemize}
\end{lem}

\begin{proof}
By induction on $\delta \leq \gamma$ we simultaneously prove that $r\re\delta$ as defined in the
lemma is indeed a condition and that $r\re\delta \leq p \re\delta , q \re\delta$ holds. 
For $\delta = 1$, this is straightforward by assumption. The limit step is also clear.

So suppose this has been proved for $\delta \geq 1$, and we show it for $\delta + 1$.
If $(\alpha,n) \in r(\delta)$ for some $\alpha$ and $n$, then, without loss of generality, we may assume
$(\alpha,n) \in p(\delta)$. Thus $p \re \delta \forces \alpha \notin \dot A_\delta$ and, since $r\re\delta \leq p\re\delta$ by
induction hypothesis, we also have $r\re\delta \forces \alpha\notin \dot A_\delta$, as
required. 

Next assume $(\alpha,n)$ and $(\sigma,n)$ belong to $r(\delta)$ for some $\alpha, \sigma$ and $n$.
If both belong to either $p(\delta)$ or $q(\delta)$, there is nothing to show. So 
we may assume without loss of generality $(\alpha,n) \in p(\delta)$ and $(\sigma,n) \in q (\delta)$.
In particular, $\delta \in \supp (p) \cap \supp (q)$ and, by assumption, $(\sigma,n) \in p(\delta)$ follows.
Since $p$ is a condition, $\sigma_\alpha^r = \sigma_\alpha^p$ and $\sigma$ are incompatible,
as required.

Thus we have proved that $r\re\delta + 1 \in \PP_{\delta + 1}$, and $r \re \delta + 1 \leq p \re\delta + 1,
q \re\delta +1$ now follows easily. This completes the induction and the proof of the lemma.
\end{proof}

This lemma presents a basic pattern of how to define a common extension
of two conditions and then show by induction that the object defined really is a condition.
We shall use this pattern several times, see Lemma~\ref{compat-lem} and Claim~\ref{q-claim} below. 
An immediate consequence is:

\begin{cor}
$\PP_\gamma$ is ccc (and even satisfies Knaster's condition) for every $\gamma \leq \kappa^+$.
\end{cor}

\begin{proof}
Let $\{ p_\zeta: \zeta < \omega_1 \} \sub \PP_\gamma$. By a straightforward $\Delta$-system argument
we see that we may assume that for any $\zeta < \xi  < \omega_1$, $p = p_\zeta$ and $q = p_\xi$
satisfy the assumptions of the previous lemma. Hence the $p_\zeta$ are pairwise compatible.
\end{proof}

Furthermore, compatible conditions sort of have a ``minimal" extension.

\begin{lem}   \label{compat-lem}
Assume $p,q \in \PP_\gamma$ are compatible with common extension $r$. Then there is a condition
$s = s^{p,q,r} \in \PP_\gamma$ with $r \leq s \leq p,q$ such that 
\begin{itemize}
\item $\supp (s) = \supp (p) \cup \supp (q)$,
\item $F^s = F^p \cup F^q$,
\item $\sigma_\alpha^s = \sigma_\alpha^r$ for all $\alpha \in F^s$,
\item $s(\delta) = p(\delta) \cup q(\delta)$ for all $\delta \in \supp (s)$ with $\delta > 0$.
\end{itemize}
\end{lem}

\begin{proof}
By induction on $\delta\leq\gamma$ we simultaneously prove that $s\re \delta$ as defined in the lemma
is indeed a condition and that $r\re\delta \leq s\re\delta \leq p\re\delta,q\re\delta$ holds.
For $\delta = 1$ this is obvious. The limit step is also straightforward.

So assume this has been proved for $\delta \geq 1$. First suppose that $(\alpha,n) \in s(\delta)$ for some
$\alpha$ and $n$. Without loss of generality we may assume that $(\alpha,n) \in p(\delta)$.
Thus $p\re\delta \forces \alpha \notin \dot A_\delta$. Since $s\re\delta \leq p\re\delta$ by induction
hypothesis, also $s \re\delta \forces \alpha \notin \dot A_\delta$, as required.

Next suppose $(\alpha,n) , (\sigma,n)$ in $s(\delta)$ for some $\alpha,\sigma$ and $n$. 
Since $r\leq p,q$, $r(\delta) \supseteq s(\delta)$, and $(\alpha,n) , (\sigma,n) \in r(\delta)$ follows.
In particular, $\sigma^s_\alpha = \sigma^r_\alpha$ and $\sigma$ are incompatible in $\twolom$,
as required. 

Thus $s\re \delta + 1 \in \PP_{\delta + 1}$ and $r \re \delta + 1 \leq s \re \delta + 1 \leq p \re \delta + 1, q \re \delta + 1$
is now obvious. This completes the induction and the proof of the lemma.
\end{proof}

In particular, two conditions $p,q$ are compatible iff they have a common extension $s$
with $\supp (s) = \supp (p) \cup \supp (q)$, $F^s = F^p\cup F^q$, $\sigma_\alpha^s \supseteq \sigma_\alpha^p, \sigma_\alpha^q$
for all $\alpha \in F^s$, and $s(\delta) = p(\delta) \cup q(\delta)$ for all $\delta \in \supp (s)$. 
(Lemma~\ref{compat-lem} will not be needed but gives some motivation for Claim~\ref{q-claim} below.)

Let $\dot G$ be the name for the $\PP_{\kappa^+}$-generic filter. 
For $0 < \gamma < \kappa^+$ and $n \in\omega$ define $\PP_{\kappa^+}$-names 
(more explicitly, $\PP_{\gamma + 1}$-names)
\[ \dot U_{\gamma,n} = \bigcup \{ [\sigma] : (\sigma,n) \in p(\gamma) \mbox{ for some } p \in \dot G \} \mbox{ and } 
\dot H_\gamma = \bigcap_{n \in\omega} \dot U_{\gamma , n}. \]
Clearly, the $\dot U_{\gamma,n}$ are names for open sets, and $\dot H_\gamma$ is the name for a $G_\delta$ set.
Also let
\[ \dot c_\alpha = \bigcup \{ \sigma : \sigma = \sigma^p_\alpha \mbox{ for some } p \in \dot G \} \]
be the canonical name for the Cohen real added in coordinate $\alpha$ of stage $1$ of the forcing,
for each $\alpha < \kappa$.

\begin{lem}   \label{Q-lem}
{\rm (i)} Assume $n \in\omega, \alpha < \kappa , \gamma < \kappa^+$ and $p \in\PP_{\kappa^+}$ are given
such that $p\re\gamma \forces_\gamma \alpha \in \dot A_\gamma$. Then there are $q \leq p$ and $\sigma
\sub \sigma_\alpha^q$ such that $(\sigma,n) \in q(\gamma)$. 

{\rm (ii)} Assume $\alpha < \kappa , \gamma < \kappa^+$ and $p \in\PP_{\kappa^+}$ are given
such that $p\re\gamma \forces_\gamma \alpha \notin \dot A_\gamma$. Then there are $q \leq p$ and $n
\in\omega$ such that $(\alpha,n) \in q(\gamma)$. 
\end{lem}

\begin{proof}
(i) Assume $n,\alpha,\gamma, p$ are given as required. Then clearly $(\alpha,n) \notin p(\gamma)$.
For $\beta \in F^p$ extend $\sigma^p_\beta$ to $\sigma_\beta^q$ such that they are pairwise incompatible. 
This defines $q\re 1$ (in particular, $F^q = F^p$). For $\delta >0$ let
\[ q(\delta) = \begin{cases} p(\delta) & \mbox{if } \delta \neq\gamma  \\
   p(\delta) \cup \{ (\sigma_\alpha^q , n) \}  & \mbox{if } \delta = \gamma \\ \end{cases} \]
(in particular, $\supp (q) = \supp (p) \cup \{\gamma \}$). It is easy to see that $q$ is a condition and that it strengthens $p$.

(ii) Assume $\alpha,\gamma,p$ are given as required. Let $F^q = F^p \cup \{ \alpha \}$ and let $\sigma_\alpha^q$
be either $\sigma_\alpha^p$ (if $\alpha \in F^p$) or arbitrary (otherwise).
Let $n$ be large enough so that no $(\sigma,n)$ belongs to $p(\gamma)$.
For $\delta >0$ let
\[ q(\delta) = \begin{cases} p(\delta) & \mbox{if } \delta \neq\gamma  \\
   p(\delta) \cup \{ (\alpha , n) \}  & \mbox{if } \delta = \gamma \\ \end{cases} \]
(in particular, $\supp (q) = \supp (p) \cup \{\gamma \}$).
Again, it is easy to see that $q$ is a condition and that it strengthens $p$.
\end{proof}

\begin{cor}
{\rm (i)} Assume $\alpha < \kappa , \gamma < \kappa^+$ and $p \in\PP_{\kappa^+}$ are such that 
$p\re\gamma \forces_\gamma \alpha \in \dot A_\gamma$. Then $p \re\gamma+1 \forces_{\gamma+1} \dot c_\alpha \in \dot H_\gamma$.

{\rm (ii)} Assume $\alpha < \kappa , \gamma < \kappa^+$ and $p \in\PP_{\kappa^+}$ are such that 
$p\re\gamma \forces_\gamma \alpha \notin \dot A_\gamma$. Then $p \re\gamma+1 \forces_{\gamma+1} \dot c_\alpha \notin \dot H_\gamma$.
\end{cor}

\begin{proof}
(i) This is immediate from Lemma~\ref{Q-lem} (i). 

(ii) Let $q \leq p$ be arbitrary. By Lemma~\ref{Q-lem} (ii), there are $r \leq q$ and $n \in \omega$ with $(\alpha,n) \in r (\gamma)$.
It suffices to show that $r \forces \dot c_\alpha \notin \dot U_{\gamma,n}$.

Assume that $G$ is $\PP_{\kappa^+}$-generic with $r \in G$ and $c_\alpha \in U_{\gamma,n}$. By definition of $U_{\gamma,n}$,
there are $r' \leq r$ and $\sigma \in \twolom$ such that $r' \in G$, $(\sigma,n) \in r'(\gamma)$ and $\sigma \sub c_\alpha$.
Hence there is $r'' \leq r'$ such that $r'' \in G$ and $\sigma \sub \sigma^{r''}_\alpha$. Since $(\alpha,n), (\sigma,n) \in r'' (\gamma)$,
this contradicts the definition of a condition. Hence we must have $c_\alpha \notin U_{\gamma,n}$, as required.
\end{proof}

As a consequence, we see that 
\[ \forces_{\gamma + 1} \dot H_\gamma \cap \{ \dot c_\alpha : \alpha < \kappa \} = \{ \dot c_\alpha : \alpha \in \dot A_\gamma \}. \]
Thus we obtain:

\begin{cor}
$\forces_{\kappa^+} `` \{ \dot c_\alpha : \alpha < \kappa \}$ is a Q-set$"$.
\end{cor}


To see that the square of the set of Cohen reals is not a Q-set, it clearly suffices to establish the following:

\begin{mainlem}   \label{nonQ-lem}
$\forces_{\kappa^+} ``\{ (\dot c_\alpha,\dot c_\beta ) : \alpha < \beta <\kappa \}$ is not a relative $G_\delta$-set$"$.
\end{mainlem}

\begin{proof}
Assume that $\dot V_n$, $n \in\omega$, are $\PP_{\kappa^+}$-names of open sets in the plane $(\twoom)^2$
such that $\{ (\dot c_\alpha, \dot c_\beta) : \alpha < \beta < \kappa \} \sub \bigcap_{n\in\omega}\dot V_n$ is forced by the
trivial condition. We shall find $\beta < \alpha < \kappa$ such that the trivial condition
forces $(\dot c_\alpha, \dot c_\beta) \in \bigcap_{n\in\omega}\dot V_n$.
This is clearly sufficient.

There are $\dot S_n \sub (\twolom)^2$ such that for every $n \in \omega$, $\dot V_n = \bigcup \{ [\sigma] \times [\tau] : (\sigma,\tau) \in
\dot S_n \}$ is forced. For each $n \in \omega$ and each $(\sigma,\tau) \in (\twolom)^2$, let
$\{ p^j_{n,\sigma,\tau} : j \in \omega \}$ be a maximal antichain of conditions in $\PP_{\kappa^+}$ deciding the
statement $(\sigma,\tau) \in \dot S_n$. 

Let $Z = \{ z_i : i \in \omega \}$ be a countable set disjoint from $\kappa$. We
say that $\sft = (\Gamma^\sft,\Delta^\sft , \bar D^\sft = (D_\gamma^\sft: \gamma \in \Gamma^\sft\sem \{ 0 \} ) , \bar E^\sft = (E^\sft_\gamma ;
\gamma \in \Gamma^\sft \sem \{ 0 \}) , \bar\tau^\sft = (\tau_\zeta^\sft : \zeta \in \Delta^\sft))$ is a {\em pattern} if 
\begin{itemize}
\item $\Gamma^\sft \sub \kappa^+$ is finite, $0 \in \Gamma^\sft$,
\item $\Delta^\sft \sub \kappa \cup Z$ is finite and $\Delta^\sft \cap Z$ is an initial segment of $Z$, \\ i.e., $\Delta^\sft \cap Z = 
   \{ z_i : i < k \}$ for some $k$,
\item $D_\gamma^\sft \sub \twolom \times \omega$ is finite for $\gamma \in \Gamma^\sft$,
\item $E_\gamma^\sft \sub \Delta^\sft \times \omega$ is finite for $\gamma \in \Gamma^\sft$,
\item $\tau_\zeta^\sft \in \twolom$ for $\zeta \in \Delta^\sft$.
\end{itemize}
Usually we will omit the superscript $\sft$. Let $\PAT$ denote the collection of all patterns. 

Let $X \sub \kappa^+$ and $Y \sub \kappa$. Assume $0\in X$. Also let $p \in \PP_{\kappa^+}$.
We say that $(\Gamma , \Delta , \bar D , \bar E, \bar \tau) \in \PAT$ is the {\em $(X,Y)$-pattern of $p$}
if $\Gamma \sub X$, $\Delta \sub Y \cup Z$ and for some (necessarily unique) one-to-one function $\varphi = \varphi^p : 
\Delta \to \kappa$ with $\varphi \re (\Delta \cap Y) = \id$ and $\varphi \re (\Delta \cap Z) : Z \to \kappa \sem Y$ order-preserving
(i.e., if $i<j$ and $z_i , z_j \in \Delta \cap Z$, then $\varphi (z_i) < \varphi (z_j)$),
the following hold:
\begin{itemize}
\item $\supp (p) \cap X = \Gamma$,
\item $F^p = \{ \varphi (\zeta) : \zeta \in \Delta \}$
   (in particular, $F^p \cap Y = \Delta \cap Y$ and $F^p \sem Y = \{ \varphi (\zeta) : \zeta \in \Delta \cap Z \}$),
\item $\sigma^p_{\varphi(\zeta)} = \tau_\zeta$ for $\zeta \in \Delta$,
\item for $\gamma \in \Gamma$ and $(\sigma,n) \in \twolom \times \omega$:
   \[ (\sigma,n) \in p (\gamma) \;\Loleriar\; (\sigma,n) \in D_\gamma, \]
\item for $\gamma \in \Gamma$ and $(\zeta,n) \in \Delta  \times \omega$:
   \[ (\varphi(\zeta), n) \in p (\gamma) \;\Loleriar\; (\zeta,n) \in E_\gamma. \]
\end{itemize}
Clearly each condition has a unique $(X,Y)$-pattern.

Let $\chi \geq (2^{\kappa^+})^+$, and let $M \prec H (\chi)$ be a countable elementary submodel containing
$\kappa$, $\PP_{\kappa^+}$, $Z$, $\PAT$, all $\dot V_n$, $\dot S_n$, and $\{ p^j_{n,\sigma,\tau} : j \in \omega, 
(\sigma,\tau) \in ( \twolom )^2 \}$.
Choose any ordinals $\beta, \alpha$ with $M \cap \omega_1 \leq \beta < \alpha < \omega_1$.
We shall see that $\beta$ and $\alpha$ are as required, that is, that $\forces (\dot c_\alpha, \dot c_\beta) \in \bigcap_{n \in\omega} \dot V_n$.
By a standard density argument, it suffices to prove the following:


\begin{mainclaim}
For all $n \in \omega$ and all $p \in \PP_{\kappa^+}$ there are $q \leq p$ and $\sigma_0, \tau_0 \in \twolom$
such that $\sigma_0 \sub \sigma_\alpha^q$, $\tau_0 \sub \sigma_\beta^q$ and $q \forces (\sigma_0 , \tau_0) \in \dot S_n$.
\end{mainclaim}

\begin{proof}
Fix $n \in\omega_2$. Let $p \in \PP_{\kappa^+}$ be given. By going over to a stronger condition, if necessary,
we may assume $\alpha, \beta \in F^p$. Let $X = M \cap \kappa^+$ and $Y = M \cap \kappa$.
Let $\sft = (\Gamma, \Delta, \bar D , \bar E, \bar\tau)$ be the $(X,Y)$-pattern of
$p$. Put $X_0 := \Gamma$ and $Y_0 := \Delta \cap \kappa$. Then we see that $X_0 = \Gamma = \supp (p) \cap X =
\supp (p) \cap M$ and $Y_0 = \Delta \cap Y = F^p \cap Y = F^p \cap M$. In particular, $\sft$ is
also the $(X_0, Y_0)$-pattern of $p$. Furthermore, unlike $X$ and $Y$, $X_0$ and $Y_0$ belong to $M$,
and so does $\sft$. Since $\alpha \in F^p \sem Y$, there is $i \in \omega$ such that $\varphi^p (z_i) = \alpha$.
So 
\[ H_\chi \models \exists p' \in \PP_{\kappa^+} \; \exists \alpha ' < \omega_1 \; (\sft \mbox{ is the } (X_0,Y_0) \mbox{-pattern of }
   p' \mbox{ and } \varphi^{p'} (z_i) = \alpha ' )\]
because this statement is true for $p ' = p$ and $\alpha ' = \alpha$.
By elementarity we obtain
\[ M \models \exists p' \in \PP_{\kappa^+} \; \exists \alpha ' < \omega_1 \; (\sft \mbox{ is the } (X_0,Y_0) \mbox{-pattern of }
   p' \mbox{ and } \varphi^{p'} (z_i) = \alpha ' ).\]
Let $p' \in M \cap \PP_{\kappa^+}$ and $\alpha ' < M \cap \omega_1 \leq \beta$ be the witnesses.

\begin{claim}
$p$ and $p'$ are compatible, with common extension $p''$ given canonically according to Lemma~\ref{compat-crit}.
\end{claim}

\begin{proof}
We need to verify the two conditions in Lemma~\ref{compat-crit}. 
Since $F^{p'} \sub M \cap \kappa = Y$ and $F^p \cap Y = F^p \cap Y_0 = \Delta \cap Y_0 = F^{p'} \cap Y_0$, we see
that $F^p \cap F^{p'} = \Delta \cap Y_0$ and $\sigma_\zeta^p = \tau_\zeta = \sigma_\zeta^{p'}$ for all $\zeta \in \Delta \cap Y_0$.
Similarly, since $\supp (p') \sub M \cap \kappa^+ = X$ and $\supp (p) \cap X = \supp (p) \cap X_0 = \Gamma = \supp (p') \cap X_0$,
we have that $\supp (p) \cap \supp (p') = \Gamma$ and, for $\gamma \in \Gamma$,
\[ (\sigma,n) \in p(\gamma) \; \Loleriar \; (\sigma,n) \in D_\gamma \; \Loleriar \; (\sigma,n) \in p' (\gamma). \]
Hence the conditions are indeed satisfied.
\end{proof}

Let $\psi = \varphi^{p'} \circ ( \varphi^p)^{-1}$. Clearly $\psi$ maps $F^p$ one-to-one and onto $F^{p'}$,
and we have $\alpha' = \psi (\alpha)$. We also note for later use that for all $\zeta \in F^p$,
$\sigma_\zeta^{p''} = \sigma_\zeta^p = \tau_{ ( \varphi^p)^{-1} (\zeta)} = \tau_{ ( \varphi^{p'})^{-1} (\psi (\zeta)) } =
\sigma_{\psi(\zeta)}^{p'} = \sigma_{\psi(\zeta)}^{p''}$, and for all $\gamma \in \Gamma$ and $\zeta \in F^p$,
\[ (\zeta,n) \in p (\gamma) \; \Loleriar \; ( (\varphi^p)^{-1} (\zeta) , n) \in E_\gamma \; \Loleriar \; ( \psi (\zeta) , n ) \in p' (\gamma). \]
Since $\forces (\dot c_{\alpha'} , \dot c_\beta) \in \dot V_n$, there are $\tilde p \leq p''$ and $\sigma_0, \tau_0 \in
\twolom$ such that
\[ \tilde p \forces (\dot c_{\alpha '} , \dot c_\beta) \in [ \sigma_0 ] \times [\tau_0] \sub \dot V_n, \]
that is, $\sigma_0 \sub \sigma_{\alpha '}^{\tilde p} , \tau_0 \sub \sigma_\beta^{\tilde p}$ and
$ \tilde p \forces (\sigma_0 , \tau_0 ) \in \dot S_n$.

By construction, for some $j \in \omega$, the condition $r = p^j_{n, \sigma_0, \tau_0} \in M$ is compatible
with $\tilde p$. In particular, $r$ must also force $(\sigma_0, \tau_0) \in \dot S_n$. Furthermore we know that
$\supp (r) \sub X$ and $F^r \sub Y$. By strengthening $\tilde p$, if necessary, we may assume that
$\tilde p \leq r$.

\begin{claim}   \label{q-claim}
$p$ and $r$ have a common extension $q$ such that $\sigma_\alpha^q = \sigma_{\alpha '}^{\tilde p}$
and $\sigma_\beta^q = \sigma_\beta^{\tilde p}$.
\end{claim}

Note that we know already that $p$ and $r$ are compatible because they have common extension $\tilde p$;
however, here we construct a different common extension $q$ (in general $\sigma_\alpha^{\tilde p}$ and $\sigma_\alpha^q$
are distinct).

\begin{proof}
As in the proofs of Lemmata~\ref{compat-crit} and~\ref{compat-lem}, we define $q$ and 
then show by induction on $\gamma$ that $q \re \gamma
\in \PP_\gamma$ and that $q \re\gamma$ extends both $p\re\gamma$ and $r\re\gamma$.
\begin{itemize}
\item $\supp (q) = \supp (p) \cup \supp (r)$,
\item $F^{q} = F^{p} \cup F^{r}$,
\item for $\zeta \in F^{q}$, $\sigma_\zeta^{q} = \begin{cases} \sigma_\zeta^{\tilde p} & \mbox{if } \zeta\in F^{r}  \mbox{ or }  \zeta = \beta \\
   \sigma_{\psi(\zeta)}^{\tilde p}  & \mbox{if } \zeta\in F^{p} \sem F^{r} \mbox{ and } \zeta \neq \beta, \\  \end{cases}$
\item for $\gamma \in \supp (q)$ with $\gamma >0$, $q(\gamma) = p(\gamma) \cup r (\gamma)$.
\end{itemize}
For the induction, first let $\gamma = 1$. If $\zeta \in F^r$, then by $\tilde p \leq r$,
$\sigma_\zeta^r \sub \sigma_\zeta^{\tilde p} = \sigma_\zeta^q$. If, additionally, $\zeta \in F^p$,
then by $\tilde p \leq p''$, $\sigma_\zeta^p = \sigma_\zeta^{p''} \sub \sigma_\zeta^{\tilde p} = \sigma_\zeta^q$.
If $\zeta = \beta$, we similarly have $\sigma_\beta^p = \sigma_\beta^{p''} \sub \sigma_\beta^{\tilde p} = \sigma_\beta^q$.
Finally, if $\zeta \in F^p \sem F^r$ with $\zeta \neq \beta$, then, using again $\tilde p \leq p''$,
we see that $\sigma_\zeta^p = \sigma_{\psi (\zeta)}^{p'} =
\sigma_{\psi (\zeta)}^{p''} \sub \sigma_{\psi (\zeta)}^{\tilde p} = \sigma_\zeta^q$. 
Thus $q \re 1 \leq p \re 1, r  \re 1$, as required.

Note, in particular, that $\sigma_\alpha^q = \sigma_{\psi (\alpha)}^{\tilde p} = \sigma_{\alpha '}^{\tilde p}$.

As usual, the limit step of the induction is trivial. Let us assume we have
proved the statement for some $\gamma \geq 1$, and let us prove it for $\gamma + 1$.
First assume that $(\zeta,n) \in q(\gamma)$. Then we see that $q\re\gamma \forces \zeta \notin \dot A_\gamma$
as in the proofs of Lemmata~\ref{compat-crit} and~\ref{compat-lem}.

Hence assume $(\zeta,n), (\sigma,n) \in q (\gamma)$. As in the proof of Lemma~\ref{compat-crit}, we may assume that $\gamma \in
\supp (p) \cap \supp (r)$. In particular $\gamma \in X_0 = \Gamma$. First suppose $(\zeta,n) \in r (\gamma)$ and
$(\sigma,n) \in p(\gamma)$. Then $\zeta \in F^r \sub F^{\tilde p}$. Since $\tilde p \leq p,r$, we see that $(\zeta,n),
(\sigma, n) \in \tilde p (\gamma)$,
and since $\tilde p$ is a condition, $\sigma$ and $\sigma_\zeta^{\tilde p} = \sigma_\zeta^q$ are incompatible,
as required. 

Next suppose $(\zeta,n) \in p(\gamma)$ and $(\sigma,n) \in r(\gamma)$. So $\zeta \in F^p$. We split into cases according to
whether $\zeta = \beta$ or not. First assume $\zeta \neq \beta$. Then, using $\tilde p \leq p'' \leq p'$,
we see $(\psi(\zeta), n) \in p' (\gamma) \sub
p'' (\gamma) \sub \tilde p (\gamma)$ and $(\sigma,n) \in \tilde p (\gamma)$. Since $\tilde p$ is a condition,
$\sigma_{\psi(\zeta)}^{\tilde p}$ and $\sigma$ are incompatible. If $\zeta \in F^r$, then $\zeta\in Y_0$, and $\psi (\zeta) = \zeta$ follows.
Hence $\sigma_\zeta^q = \sigma_\zeta^{\tilde p} = \sigma_{\psi (\zeta)}^{\tilde p}$.
If $\zeta \notin F^r$, then $\sigma_\zeta^q = \sigma_{\psi (\zeta)}^{\tilde p}$ by definition.
In either case, we see that $\sigma_\zeta^q$ and $\sigma$ are incompatible.

Now assume $\zeta = \beta$. Since $\tilde p \leq p,r$, we have $(\beta,n) , (\sigma,n) \in \tilde p (\gamma)$,
and $\sigma_\beta^q = \sigma_\beta^{\tilde p}$ and $\sigma$ are incompatible because $\tilde p$ is a
condition.

This shows that $q \re \gamma + 1$ is a condition. Clearly $q \re\gamma + 1 \leq p \re \gamma + 1, r \re \gamma + 1$.
This completes the induction.
\end{proof}

Thus $q\leq p$, $\sigma_0 \sub \sigma_\alpha^q$, $\tau_0 \sub \sigma_\beta^q$ and $q$ forces that $(\sigma_0,\tau_0)$
belongs to $\dot S_n$, and the proof of the main claim is complete.
\end{proof}

This completes the proof of the main lemma and of the theorem.
\end{proof}


\begin{thebibliography}{ABC}



\bibitem[Fl]{Fl83} W.G. Fleissner,
{\em Squares of $Q$ sets},
Fund. Math. {\bf 118} (1983), 223-231.

\bibitem[FM]{FM80} W.G. Fleissner and A.W. Miller,
{\em On $Q$ sets},
Proc. Amer. Math. Soc. {\bf 78} (1980), 280-284.

\bibitem[MS]{MS70} D.A. Martin and R.M. Solovay,
{\em Internal Cohen extensions},
Ann. Math. Logic {\bf 2} (1970), 143-178.

\bibitem[Mi1]{Mi95} A.W. Miller,
{\em Descriptive Set Theory and Forcing},
Lecture Notes in Logic {\bf 4}, Springer 1995.

\bibitem[Mi2]{Mi07} A.W. Miller,
{\em A hodgepodge of sets of reals},
Note di Matematica {\bf 27} (2007), 25-39.

\bibitem[Mi3]{Mi15} A.W. Miller,
{\em Some interesting problems}, version of April 2015.

\bibitem[Pr]{Pr80} T. Przymusi\'nski,
{\em The existence of $Q$-sets is equivalent to the existence of strong $Q$-sets},
Proc. Amer. Math. Soc. {\bf 79} (1980), 626-628.

\end{thebibliography}
\end{document}